\documentclass[a4paper]{article}

\usepackage[english]{babel}
\usepackage[utf8x]{inputenc}
\usepackage[T1]{fontenc}
\usepackage{amsmath,amsthm,amsfonts,amssymb}
\usepackage{enumerate}
\usepackage[colorinlistoftodos]{todonotes}
\usepackage{color}

\usepackage[a4paper,top=3cm,bottom=2cm,left=3cm,right=3cm,marginparwidth=1.75cm]{geometry}

\newcommand*{\m}[1]{\underline{#1}}
\DeclareMathOperator{\Cl}{\mathcal{C}\ell}

\newcommand*{\BAR}[1]{\overline{#1}}

\newcommand{\R}{\mathbb R}

\newcommand{\HH}{\mathbb H}
\newcommand{\OO}{\mathbb O}

\newcommand{\cro}{D_x}
\newcommand{\crovec}{D_{\m{x}}}
\newcommand{\parder}[2]{\partial_{x_#1}f_#2}

\newtheorem{thm}{Theorem}[section]
\newtheorem{cor}[thm]{Corollary}
\newtheorem{lem}[thm]{Lemma}
\newtheorem{prop}[thm]{Proposition}
\theoremstyle{definition}

\newtheorem{rem}[thm]{Remark}

\title{On Structure of Octonion Regular Functions}

\author{Janne Kauhanen\\
Mathematics\\
Tampere University\\
FI-33014 Tampere University\\
Finland\\
\texttt{janne.kauhanen@tuni.fi}
\and
Heikki Orelma\\
Civil Engineering\\
Tampere University\\
FI-33014 Tampere University\\
Finland\\
\texttt{heikki.orelma@tuni.fi}
}

\begin{document}
\maketitle

\begin{abstract}
In this paper we study octonion regular functions and the structural differences between regular functions in octonion, quaternion, and Clifford analyses.
\end{abstract}

\textbf{Mathematics Subject Classification (2010)}. 30G35, 15A63

\textbf{Keywords}. Octonions, Cauchy-Riemann operators, regular functions

\section{Introduction}
In our recent papers \cite{KO1,KO2}, we started to study octonion algebraic methods in analysis. This paper is a continuation of our studies in this fascinating field. Over  the years, many results of octonion analysis have been published and studied since the fundamental paper of Dentoni and Sce \cite{DS}. One thing which has remained unclear to us is that what is octonion analysis all about? A consensus has been that octonion, quaternion and Clifford analyses are similar from a theoretical point of view, and maybe for this reason octonion analysis has been left to less attention. Our aim is to prove that octonion analysis and Clifford analysis are
 different theories from the point of view of regular functions. Thus, octonion analysis is a completely independent research topic.

We start by recalling preliminaries of octonions and Clifford numbers and their connections via triality. We define our fundamental function classes, i.e., left-, right- and bi-regular functions. We give chararacterizations for function classes in biaxial quaternion analysis and in Clifford analysis. The classical Riesz system of Stein and Weiss is used as a familiar reference to clearly see the differences.

The topic of this paper is highly technical, but we have tried to write everything in as a simple way as possible. Hopefully we have succeeded in this job. Many questions remain open and the reader may find a lot of open research problems between the lines. Hopefully we can answer some of these questions when the saga continues.

\section{Preliminaries: Octonion and Clifford algebras}
In this algebraic part of the paper, we first recall briefly the basic definitions and notations related to the octonion and Clifford algebras. Then we study their connections in detail. In the whole paper, our principle is to consider the standard orthonormal basis $\{ e_0,e_1,\ldots,e_7\}$ for $\mathbb{R}^8$, and by defining different products between the elements, we obtain different algebras. We will denote the octonion product by $e_i\circ e_j$, and the Clifford product by $e_ie_j$. 

\subsection{Octonions}
The algebra of octonions $\OO$ is the non-commutative and non-associative $8$-dimensional algebra with the basis $\{1,e_1,\ldots,e_7\}$ and the multiplication given by the following table.
\begin{center}
\begin{tabular}{c|cccccccc}
$\circ$ & $1$   & $e_1$  & $e_2$  & $e_3$  & $e_4$  & $e_5$  & $e_6$  & $e_7$\\
\hline
$1$     & $1$   & $e_1$  & $e_2$  & $e_3$  & $e_4$  & $e_5$  & $e_6$  & $e_7$\\
$e_1$   & $e_1$ & $-1$ & $e_3$  & $-e_2$ & $e_5$  & $-e_4$ & $-e_7$ & $e_6$\\
$e_2$   & $e_2$ & $-e_3$ & $-1$ & $e_1$  & $e_6$  & $e_7$  & $-e_4$ & $-e_5$\\
$e_3$   & $e_3$ & $e_2$  & $-e_1$ & $-1$ & $e_7$  & $-e_6$ & $e_5$  & $-e_4$\\
$e_4$   & $e_4$ & $-e_5$ & $-e_6$ & $-e_7$ & $-1$   & $e_1$  & $e_2$  & $e_3$\\
$e_5$   & $e_5$ & $e_4$  & $-e_7$ & $e_6$  & $-e_1$ & $-1$ & $-e_3$ & $e_2$\\
$e_6$   & $e_6$ & $e_7$ & $e_4$  & $-e_5$ & $-e_2$ & $e_3$ & $-1$  & $-e_1$\\
$e_7$   & $e_7$ & $-e_6$ & $e_5$  & $e_4$ & $-e_3$ & $-e_2$ & $e_1$  & $-1$
\end{tabular}
\end{center}
Let us point out that there are $480$ possible ways to define an octonion product such that $e_0=1$. Our choise is historically maybe the most used and traditional, and for this reason we may call it the canonical one, but e.g. Lounesto uses a different definition for octonion multiplication in his famous book \cite{L}.

For $1\le i,j\le 7$ we have
\begin{equation*}
e_i\circ e_i=e_i^2=-1,\qquad\text{and}\qquad
e_i\circ e_j=-e_j\circ e_i\quad\text{if }i\ne j.
\end{equation*}
An element $x\in\mathbb{O}$ may be represented in the forms
\begin{align*}
x&=x_0+x_1e_1+x_2e_2+x_3e_3+x_4e_4+x_5e_5+x_6e_6+x_7e_7\\
&=x_0+\m{x}\\
&=(x_0+x_1e_1+x_2e_2+x_3e_3)+(x_4+x_5e_1+x_6e_2+x_7e_3)\circ e_4\\
&=u+v\circ e_4=(u_0+\m{u})+(v_0+\m{v})\circ e_4.
\end{align*}
Here, $x_0,...,x_7\in\mathbb{R}$, $x_0$ is the \emph{real part}, \m{x} is the \emph{vector part}, and $u$ and $v\in\HH$ are quaternions. The last form is called the \emph{quaternion form} of an octonion. The \emph{conjugate} of $x$ is denoted and defined by $\BAR{x}=x_0-\m{x}$.
We see that the element $e_4$ plays a kind of a role of the ''imaginary unit''.
The product of two octonions can be written as
\begin{align}
x\circ y&=(x_0+\m{x})\circ(y_0+\m{y})=\sum_{i,j=0}^7x_iy_je_i\circ e_j
=\sum_{i=0}^7x_iy_ie_i^2+\mathop{\sum_{i,j=0}^7}_{i\ne j}x_iy_je_i\circ e_j\nonumber\\
&=x_0y_0-\sum_{i=1}^7x_iy_i+x_0\sum_{i=1}^7y_ie_i+y_0\sum_{i=1}^7x_ie_i
+\mathop{\sum_{i,j=1}^7}_{i\ne j}x_iy_je_i\circ e_j\nonumber\\
&=x_0y_0-\m{x}\cdot\m{y}
+x_0\m{y}+y_0\m{x}+\m{x}\times\m{y},\label{eq:octproddecomp}
\end{align}
where $\m{x}\cdot\m{y}$ is the \emph{dot product} and $\m{x}\times\m{y}$ the \emph{cross product} of vectors $\m{x}$ and $\m{y}$.

Denote the quaternion forms of octonions $x$ and $y$ by
\begin{equation}\label{eq:quarepxy}
\begin{aligned}
x&=(u_0+\m{u})+(v_0+\m{v})\circ e_4,\\
y&=(a_0+\m{a})+(b_0+\m{b})\circ e_4.
\end{aligned}
\end{equation}
In Lemma \ref{lem:crossquaform} we will return the cross product $\m{x}\times\m{y}$ of octonion vector parts $\m{x}$ and $\m{y}$ to cross products of the vector parts $\m{u}$, $\m{v}$, $\m{a}$, and $\m{b}$ of quaternions, which are classical $3$-dimensional cross products (see, e.g., \cite{GurlebeckSprozig,L,P}).

\begin{lem}[{See, e.g., \cite[Lemma 2.10]{KO1}}]\label{lem:quaoctprod}
Let $u,v\in\mathbb{H}$. Then
\begin{align*}
e_4\circ u&=\overline{u}\circ e_4,\\
e_4\circ (u\circ e_4)&=-\overline{u},\\
(u\circ e_4)\circ e_4&=-u,\\
u\circ (v\circ e_4)&=(v\circ u)\circ e_4,\\
(u\circ e_4)\circ v&=(u\circ \overline{v})\circ e_4,\\
(u\circ e_4)\circ (v\circ e_4)&=-\overline{v}\circ u.
\end{align*}
\end{lem}

\begin{lem}\label{lem:crossquae4}
If $x,y\in\mathbb{O}$ be as in (\ref{eq:quarepxy}), then
\begin{align*}
\m{u}\times e_4&=\m{u}\circ e_4,\\
e_4\times\m{a}&=-\m{a}\circ e_4,\\
\m{u}\times(\m{b}\circ e_4)&=-(\m{u}\times\m{b})\circ e_4-(\m{u}\cdot\m{b})e_4,\\
e_4\times(\m{b}\circ e_4)&=\m{b},\\
(\m{v}\circ e_4)\times\m{a}&=-(\m{v}\times\m{a})\circ e_4+(\m{a}\cdot\m{v})e_4,\\
(\m{v}\circ e_4)\times e_4&=-\m{v},\\
(\m{v}\circ e_4)\times(\m{b}\circ e_4)&=-\m{v}\times\m{b}.
\end{align*}
\end{lem}

\begin{proof}
Lemma \ref{lem:quaoctprod} implies
\begin{align*}
e_i\circ(e_j\circ e_4)&=(e_j\circ e_i)\circ e_4,\\
e_4\circ(e_j\circ e_4)&=e_j,\\
(e_i\circ e_4)\circ e_j&=-(e_i\circ e_j)\circ e_4,\\
(e_i\circ e_4)\circ e_4&=-e_i,\\
(e_i\circ e_4)\circ(e_j\circ e_4)&=e_j\circ e_i
\end{align*}
for $1\le i,j\le3$. Then
\begin{align*}
\m{u}\times(\m{b}\circ e_4)
&=\sum_{i,j=1}^3u_ib_je_i\circ(e_j\circ e_4)
=\left(\sum_{i,j=1}^3u_ib_je_j\circ e_i\right)\circ e_4\\
&=\left(-\mathop{\sum_{i,j=1}^3}_{i\ne j}u_ib_je_i\circ e_j-\sum_{i=1}^3u_ib_i\right)\circ e_4
=-(\m{u}\times\m{b})\circ e_4-(\m{u}\cdot\m{b})e_4,
\end{align*}
\begin{align*}
e_4\times(\m{b}\circ e_4)
=\sum_{j=1}^3b_je_4\circ(e_j\circ e_4)
=\sum_{j=1}^3b_je_j
=\m{b},
\end{align*}
\begin{align*}
(\m{v}\circ e_4)\times\m{a}
&=\sum_{i,j=1}^3v_ia_j(e_i\circ e_4)\circ e_j
=\left(-\sum_{i,j=1}^3v_ia_je_i\circ e_j\right)\circ e_4\\
&=\left(-\mathop{\sum_{i,j=1}^3}_{i\ne j}v_ia_je_i\circ e_j+\sum_{i=1}^3v_ia_i\right)\circ e_4
=-(\m{v}\times\m{a})\circ e_4+(\m{a}\cdot\m{v})e_4,
\end{align*}
\begin{align*}
(\m{v}\circ e_4)\times e_4
=\sum_{i=1}^3v_i(e_i\circ e_4)\circ e_4
=-\sum_{i=1}^3v_ie_i=-\m{v},
\end{align*}
and
\begin{align*}
(\m{v}\circ e_4)\times(\m{b}\circ e_4)
&=\mathop{\sum_{i,j=1}^3}_{i\ne j}v_ib_j(e_i\circ e_4)\circ(e_j\circ e_4)
=\mathop{\sum_{i,j=1}^3}_{i\ne j}v_ib_je_j\circ e_i
=-\mathop{\sum_{i,j=1}^3}_{i\ne j}v_ib_je_i\circ e_j\\
&=-\m{v}\times\m{b}.\qedhere
\end{align*}
\end{proof}


\begin{lem}\label{lem:crossquaform}
Denote the quaternion representations of the vectors $\m{x}$ and $\m{y}\in\OO$ as in \eqref{eq:quarepxy}. Then the cross product in quaternion form is
\begin{equation*}
\begin{aligned}
\m{x}\times\m{y}
&=v_0\m{b}-\m{v}b_0+\m{u}\times\m{a}-\m{v}\times\m{b} && \in\operatorname{span}\{e_1,e_2,e_3\}\\
&\quad +(\m{v}\cdot\m{a}-\m{u}\cdot\m{b})e_4 && \in\operatorname{span}\{e_4\}\\
&\quad +(\m{u}b_0-v_0\m{a}-\m{u}\times\m{b}-\m{v}\times\m{a})\circ e_4 && \in\operatorname{span}\{e_5,e_6,e_7\}
\end{aligned}
\end{equation*}
\end{lem}

\begin{proof}
By Lemma \ref{lem:crossquae4}, we compute
\begin{align*}
\m{x}\times\m{y}
&=\m{u}\times\m{a}+\m{u}\times(b_0e_4)+\m{u}\times(\m{b}\circ e_4)\\
&\quad +(v_0e_4)\times\m{a}+(v_0e_4)\times(b_0e_4)+(v_0e_4)\times(\m{b}\circ e_4)\\
&\quad +(\m{v}\circ e_4)\times\m{a}+(\m{v}\circ e_4)\times(b_0e_4)+(\m{v}\circ e_4)\times(\m{b}\circ e_4)\\
&=\m{u}\times\m{a}+\m{u}b_0\circ  e_4-(\m{u}\times\m{b})\circ e_4\\
&\quad -v_0\m{a}\circ e_4+0+v_0\m{b}\\
&\quad -(\m{v}\times\m{a})\circ e_4-\m{v}b_0-\m{v}\times\m{b}\\
&=\m{u}\times\m{a}-\m{v}\times\m{b}+v_0\m{b}-\m{v}b_0
\\
&\quad +(\m{v}\cdot\m{a}-\m{u}\cdot\m{b})e_4
\\
&\quad +(\m{u}b_0-v_0\m{a}-\m{u}\times\m{b}-\m{v}\times\m{a})\circ e_4.\qedhere
\end{align*}
\end{proof}


\begin{cor}\label{cor:octprodqua}
If $x,y\in\mathbb{O}$ be as in (\ref{eq:quarepxy}), then
\begin{equation*}
\begin{aligned}
x\circ y=&u_0a_0-v_0b_0-\m{u}\cdot\m{a}-\m{v}\cdot\m{b}&&\in\R\\
&+u_0\m{a}+a_0\m{u}+v_0\m{b}-\m{v}b_0+\m{u}\times\m{a}-\m{v}\times\m{b}&&\in\operatorname{span}\{e_1,e_2,e_3\}\\
&+(u_0b_0+a_0v_0+\m{v}\cdot\m{a}-\m{u}\cdot\m{b})e_4&&\in\operatorname{span}\{e_4\}\\
&+(u_0\m{b}+a_0\m{v}+\m{u}b_0-v_0\m{a}-\m{u}\times\m{b}-\m{v}\times\m{a})\circ e_4&&\in\operatorname{span}\{e_5,e_6,e_7\}
\end{aligned}
\end{equation*}
\end{cor}

\subsection{The Clifford algebra $\Cl_{0,7}$ and triality}
Since the dimension of octonions and Clifford paravectors is $8$, they behave similarly as vector spaces. Moreover, we may ask if there is a connection between the octonion and the Clifford product? The answer is given by Pertti Lounesto in his book \cite{L}. We will recall his ideas here in detail. Let us recall the basic definitions and properties of Clifford algebras.

 We continue working with the basis $\{ e_0,e_1,...,e_7\}$ for $\mathbb{R}^8$. The Clifford product is defined by
\[
e_ie_j+e_je_i=-2\delta_{ij},\ i,j=1,...,7,
\]
where $\delta_{ij}$ is the Kronecker delta symbol. Here, $e_0=1$. Then, similarly than in the case of octonions, $e_0^2=1$, and $e^2_j=-1$ for all $j=1,\ldots,7$. The Clifford product $e_ie_j$ is not necessary a vector or a scalar. 
This product generates an associative algebra, called 
the Clifford algebra, denoted by $\Cl_{0,7}$. The dimension of this Clifford algebra is $2^7$, and an element $a\in \Cl_{0,7}$ may be represented as a sum
\[
a=\sum_{j=0}^7[a]_j
\]
of a scalar part $[a]_0$, generated by $1$, a $1$-vector part $[a]_1$, generated by $e_j$'s, $2$-vector part $[a]_2$, generated by the products $e_ie_j$, where $1\le i<j\le7$, etc.
Clifford numbers of the form $[a]_1$ are called \emph{vectors} and $[a]_{0,1}=[a]_0+[a]_1$ \emph{paravectors}. The set of paravectors may be identified with $\mathbb{R}^8$.

The Clifford product of two paravectors $x$ and $y$ can be written
\begin{align}
xy&=(x_0+\m{x})(y_0+\m{y})=\sum_{i,j=0}^7x_iy_je_ie_j\nonumber\\
&=\sum_{i=0}^7x_iy_ie_i^2+\mathop{\sum_{i,j=0}^7}_{i\ne j}x_iy_je_ie_j\nonumber\\
&=x_0y_0-\sum_{i=1}^7x_iy_i+x_0\sum_{i=1}^7y_ie_i+y_0\sum_{i=1}^7x_ie_i
+\mathop{\sum_{i,j=1}^7}_{i\ne j}x_iy_je_ie_j\nonumber\\
&=x_0y_0-\m{x}\cdot\m{y}
+x_0\m{y}+y_0\m{x}+\m{x}\wedge\m{y},\label{eq:clparavect}
\end{align}
where $\m{x}\wedge\m{y}$ is the \emph{wedge product} of vectors $\m{x}$ and $\m{y}$.
In particular, $\m{x}\m{y}=\m{x}\wedge\m{y}-\m{x}\cdot\m{y}$.

The reader can see that formally the octonion and the Clifford products are similar, and a reasonable question is, how they are connected?
We would like to construct the octonion product using the Clifford algebra $\Cl_{0,7}$. Let us consider the octonion product of the basis elements $e_i$ and $e_j$, 
where $1\le i,j\le7$, $i\ne j$:
\[
e_i\circ e_j=e_k.
\]
Then $1\le k\le7$, and $i\ne k\ne j$. The corresponding Clifford product $e_ie_j$ may be mapped to $e_k$ by multiplying it by the trivector $e_je_ie_k$, i.e., 
\[
(e_ie_j)(e_je_ie_k)=e_i^2e_j^2e_k=e_k,
\]
and by the same trivector $e_je_i$ is mapped to $-e_k$. If $\m{a}$ and $\m{b}$ are vectors, 
then
\[
\m{a}\,\m{b}(e_je_ie_k)=(a_ib_j-a_jb_i)e_k+[\m{a}\,\m{b}(e_je_ie_k)]_3+[\m{a}\,\m{b}(e_je_ie_k)]_5.
\]
Picking the $1$-vector part
\[
[\m{a}\,\m{b}(e_je_ie_k)]_1=(a_ib_j-a_jb_i)e_k
\]
gives us a part of the $k$th component of the octonion product $\m{a}\circ\m{b}$. Using this idea, we may express the octonion product $a\circ b$ as the paravector part of the Clifford product
$ab(1-W)$, where $W$ is a suitable $3$-vector.

\begin{lem}[{\cite[Sec 23.3]{L}, \cite[Lem 4.1]{Auli}}]\label{lem:abW}
Define
\[
W=e_{123}+e_{145}+e_{176}+e_{246}+e_{257}+e_{347}+e_{365}.
\]
Let $a=a_0+\m{a}$ and $b=b_0+\m{b}$ be paravectors. Then
\[
a\circ b=[ab(1-W)]_{0,1}
\]
and in particular, $\m{a}\times\m{b}=-[(\m{a}\wedge\m{b})W]_1$.
\end{lem}

Lounesto states Lemmas \ref{lem:abW} and \ref{lem:16abi} without proofs at pages 303--304 in \cite{L}, and for a different multiplication table of octonions. Ven\"al\"ainen gives a proof for Lemma \ref{lem:abW} in her licentiate thesis \cite{Auli}. For the convenience of the reader, we give a short proof of Lemma \ref{lem:abW} here.

\begin{proof}
We compute
\[
[ab(1-W)]_{0,1}
=[ab]_{0,1}-[abW]_1
=a_0b_0-\m{a}\cdot\m{b}+a_0\m{b}+b_0\m{a}-[abW]_1.
\]
By \eqref{eq:octproddecomp} and \eqref{eq:clparavect}, it is enough to show that $\m{a}\times\m{b}=-[(\m{a}\wedge\m{b})W]_1$. Consider the triplets
\begin{equation*}
\nu=123, 145, 176, 246, 257, 347, 365.
\end{equation*}
The product $e_ie_je_\nu$ is a vector only if the pair of indices $ij$ belongs to the triplet $\nu$. Since the cross and the wedge products
\[
\m{a}\times\m{b}=\mathop{\sum_{i,j=1}^7}_{i<j}(a_ib_j-a_jb_i)e_i\circ e_j
\quad\text{and}\quad
\m{a}\wedge\m{b}=\mathop{\sum_{i,j=1}^7}_{i<j}(a_ib_j-a_jb_i)e_ie_j
\]
have the same coefficients, and each pair $ij$, $1\le i<j\le7$, is contained in exatly one of the triplets $\nu$, say $\nu_0$, it is enough to check that $e_ie_je_{\nu_0}=-e_i\circ e_j$ for all such pairs $ij$.
\end{proof}
A straightforward computation shows that the trivector $W$ is invertible. 
\begin{lem}
$W^{-1}=\dfrac17(W-6e_{12\cdots7})$
\end{lem}
In the above, we identify octonions $\mathbb{O}$ with the $8$-dimensional paravectors. The dimension $8$ plays a special role in the theory of spin groups, since $\operatorname{Spin}(8)$ has the so called \emph{exceptional automorphisms}. This feature is called \emph{triality}, and the first time it was noticed in the book of Study \cite{S}. For modern references, see \cite{H,L,P}. The triality means that in addition to paravectors, we may 
identify octonions with the spinor spaces $S^{\pm}$.
The spinor spaces may be realized by
\[
S^{\pm}=\Cl_{0,7}I^{\pm},
\]
where $I^{\pm}$ is a primitive idempotent,
\begin{equation}\label{eq:primidef}
I^{\pm}=\frac{1}{16}(1+We_{12\cdots7})(1\pm e_{12\cdots7}),
\end{equation}
see \cite{DSS,L}. A straightforward computation shows that the octonion product in spinor spaces may be computed as follows.


\begin{lem}[{\cite[Sec 23.3]{L}}]\label{lem:16abi}
For paravectors $a$ and $b$ we have
\begin{equation}
a\circ b=16[abI^{-}]_{0,1}.
\end{equation}
\end{lem}

\section{Octonion analysis}\label{sec:octanal}

In this section we recall the basic facts of octonion analysis, i.e., the theory of Cauchy-Riemann operators in the octonionic setting. After that, we carefully study the general structure of the null solutions of these operators and define three different classes of regular functions: left-, right-, B-, and R-regular functions. R-regular functions are just  solutions of the classical Riesz system. We use the Riesz system here as a familiar reference to better understand the structure of octonion regular functions. In Clifford analysis the corresponding function classes are equal. This structural difference is a fundamental difference between octonion and Clifford analyses.

\subsection{Cauchy-Riemann operators}

A function $f\colon\R^8\to\OO$ is of the form $f=f_0+f_1e_1+\cdots+f_7e_7=f_0+\m{f}$, where $f_j\colon\R^8\to\R$. We define the \emph{Cauchy--Riemann operator}
\[
\cro=\partial_{x_0}+e_1\circ\partial_{x_1}+...+e_7\circ\partial_{x_7}.
\]
The vector part of it
\[
\crovec=e_1\circ\partial_{x_1}+...+e_7\circ\partial_{x_7}
\]
is called the \emph{Dirac operator}. If the coordinate functions of $f$ have partial derivatives, then $\cro$ operates on $f$ from the left and from the right as
\[
\cro f
=\sum_{i,j=0}^7e_i\circ e_j\partial_{x_i}f_j
\quad\text{and}\quad
f\cro
=\sum_{i,j=0}^7e_j\circ e_i\partial_{x_i}f_j.
\]
Decomposition \eqref{eq:octproddecomp} gives (see \cite{LiPeng})
\begin{align}
\cro f&=\partial_{x_0}f_0-\crovec\cdot\m{f}+\partial_{x_0}\m{f}+\crovec f_0+\crovec\times \m{f}\quad\text{and}\label{eq:dxfdecomp}\\
f\cro&=\partial_{x_0}f_0-\crovec\cdot\m{f}+\partial_{x_0}\m{f}+\crovec f_0-\crovec\times \m{f},\label{eq:fdxdecomp}
\end{align}
where $\partial_{x_0}f_0-\crovec\cdot\m{f}$ is the \emph{divergence} of $f$ and $\crovec\times \m{f}$ is the \emph{rotor} of $\m{f}$.

If $\cro f=0$ (resp. $f\cro=0$), then $f$ is called \emph{left} (resp. \emph{right}) \emph{regular}.

In Clifford analysis one studies functions $f\colon\mathbb{R}^8\to\Cl_{0,7}$. We define the Cauchy-Riemann operator similarly as in octonion analysis:
\[
\partial_x=\partial_{x_0}+e_1\partial_{x_1}+...+e_7\partial_{x_7}=\partial_{x_0}+\partial_{\m{x}}.
\]
Functions satisfying $\partial_xf=0$ (resp. $f\partial_x=0$) on $\R^8$ are called \emph{left} (resp. \emph{right}) \emph{monogenic}. In this paper we only need to consider paravector valued functions
\[
f=f_0+f_1e_1+...+f_7e_7.
\]

\subsection{Left-, Right-, B- and R- regular functions}
\label{sec:MTBR}
Comparing the real and vector parts in \eqref{eq:dxfdecomp} and \eqref{eq:fdxdecomp} yields the following well known results.

\begin{prop}\label{prop:mtregular}
A function $f\colon\R^8\to\OO$ is left regular if and only if it satisfies the \emph{Moisil--Teodorescu type system}
\begin{equation}\label{eq:mtsystem}
\begin{aligned}
\partial_{x_0}f_0-\crovec\cdot\m{f}&=0,\\
\partial_{x_0}\m{f}+\crovec f_0+\crovec\times \m{f}&=0,
\end{aligned}
\end{equation}
or componentwise
\begin{equation}\label{eq:mtsystemcomp}
\begin{aligned}
\partial_{x_0}f_0-\partial_{x_1}f_1-\ldots-\partial_{x_7}f_7=0,\\
\parder{0}{1}+\parder{1}{0}
+\parder{2}{3}-\parder{3}{2}
+\parder{4}{5}-\parder{5}{4}
-\parder{6}{7}+\parder{7}{6}=0,\\
\parder{0}{2}+\parder{2}{0}
-\parder{1}{3}+\parder{3}{1}
+\parder{4}{6}-\parder{6}{4}
+\parder{5}{7}-\parder{7}{5}=0,\\
\parder{0}{3}+\parder{3}{0}
+\parder{1}{2}-\parder{2}{1}
+\parder{4}{7}-\parder{7}{4}
-\parder{5}{6}+\parder{6}{5}=0,\\
\parder{0}{4}+\parder{4}{0}
-\parder{1}{5}+\parder{5}{1}
-\parder{2}{6}+\parder{6}{2}
-\parder{3}{7}+\parder{7}{3}=0,\\
\parder{0}{5}+\parder{5}{0}
+\parder{1}{4}-\parder{4}{1}
-\parder{2}{7}+\parder{7}{2}
+\parder{3}{6}-\parder{6}{3}=0,\\
\parder{0}{6}+\parder{6}{0}
+\parder{1}{7}-\parder{7}{1}
+\parder{2}{4}-\parder{4}{2}
-\parder{3}{5}+\parder{5}{3}=0,\\
\parder{0}{7}+\parder{7}{0}
-\parder{1}{6}+\parder{6}{1}
+\parder{2}{5}-\parder{5}{2}
+\parder{3}{4}-\parder{4}{3}=0.
\end{aligned}
\end{equation}
\end{prop}

We will denote the space of left regular functions by $\mathcal{M}^{(\ell)}$, and similary, right regular functions by $\mathcal{M}^{(r)}$.

\begin{prop}\label{prop:rregular}
A function $f\colon\R^8\to\OO$ is both left and right regular if and only if it satisfies the system
\begin{equation}\label{eq:rieszoct}
\begin{aligned}
\partial_{x_0}f_0-\crovec\cdot\m{f}&=0,\\
\partial_{x_0}\m{f}+\crovec f_0&=0,\\
\crovec\times \m{f}&=0,
\end{aligned}
\end{equation}
or componentwise
\begin{equation}\label{eq:rieszoctcomp}
\begin{aligned}
\partial_{x_0}f_0-\partial_{x_1}f_1-\ldots-\partial_{x_7}f_7&=0,\\
\partial_{x_0}f_i+\partial_{x_i}f_0&=0,\quad i=1,\ldots,7,\\
\parder{2}{3}-\parder{3}{2}
+\parder{4}{5}-\parder{5}{4}
-\parder{6}{7}+\parder{7}{6}&=0,\\
-\parder{1}{3}+\parder{3}{1}
+\parder{4}{6}-\parder{6}{4}
+\parder{5}{7}-\parder{7}{5}&=0,\\
\parder{1}{2}-\parder{2}{1}
+\parder{4}{7}-\parder{7}{4}
-\parder{5}{6}+\parder{6}{5}&=0,\\
-\parder{1}{5}+\parder{5}{1}
-\parder{2}{6}+\parder{6}{2}
-\parder{3}{7}+\parder{7}{3}&=0,\\
\parder{1}{4}-\parder{4}{1}
-\parder{2}{7}+\parder{7}{2}
+\parder{3}{6}-\parder{6}{3}&=0,\\
\parder{1}{7}-\parder{7}{1}
+\parder{2}{4}-\parder{4}{2}
-\parder{3}{5}+\parder{5}{3}&=0,\\
-\parder{1}{6}+\parder{6}{1}
+\parder{2}{5}-\parder{5}{2}
+\parder{3}{4}-\parder{4}{3}&=0.
\end{aligned}
\end{equation}
\end{prop}

We will call functions satisfying \eqref{eq:rieszoctcomp} \emph{B-regular}, and denote the space of such functions by $\mathcal{M}_B$.  Naturally
\[
\mathcal{M}_B=\mathcal{M}^{(\ell)}\cap \mathcal{M}^{(r)}. 
\]

The fundamental difference between octonion and Clifford analyses is that in Clifford analysis the paravector valued null solutions to the Cauchy-Riemann operator satisfies the Riesz system and are at the same time left and right monogenic, which is not true in octonion analysis. The following well-known proposition follows from the definitions similarly as in octonion analysis by comparing the scalar parts, $1$-vector parts, and $2$-vector parts.

\begin{prop}\label{prop:kerdcl}
Suppose $f\colon\mathbb{R}^8\to\Cl_{0,7}$ is a paravector valued function. Then $\partial_xf=0$ if and only if $f\partial_x=0$, and this is equivalent to $f$ satisfying the \emph{Riesz-system}
\begin{equation}\label{eq:riesz}
\begin{aligned}
\partial_{x_0} f_0-\partial_{\m{x}}\cdot \m{f}&=0,\\
\partial_{x_0}\m{f}+\partial_{\m{x}}f_0&=0,\\
\partial_{\m{x}}\wedge\m{f}&=0,
\end{aligned}
\end{equation}
or componentwise
\begin{equation}\label{eq:rieszcomp}
\begin{aligned}
\partial_{x_0}f_0-\partial_{x_1}f_1-\ldots-\partial_{x_7}f_7&=0,\\
\partial_{x_0}f_i+\partial_{x_i}f_0&=0,\quad i=1,\ldots,7,\\
\partial_{x_i}f_j-\partial_{x_j}f_i&=0,\quad i,j=1,\ldots,7,\ i\ne j,
\end{aligned}
\end{equation}
\end{prop}
Functions satisfying \eqref{eq:rieszcomp} are called \emph{R-regular}, and the space of such functions is denoted by $\mathcal{M}_R$.




To convince the reader about the existence of
these function classes, we recall the following classical method from Clifford analysis.

\begin{rem}[Cauchy--Kovalevskaya extension]
If $f\colon\Omega\to\mathbb{R}$ is a real analytic function defined on an open
$\Omega\subset\mathbb{R}^7\cong\mathbb{O}\cap\{ x_0=0\}$ we may construct its Cauchy-Kovalevskaya extension analogously to Clifford analysis (see \cite{DSS}) by defining
\[
\text{CK}[f](x)=e^{-x_0 D_{\m{x}}}f(\m{x}).
\]
It is easy to see that since $f$ is real valued, $D_x\text{CK}[f]=\text{CK}[f]D_x=0$, i.e., $\text{CK}[f]\in\mathcal{M}_B$. Since $\mathbb{O}$ is an alternative division algebra, that is $x(xy)=x^2y$ for all $x,y\in\mathbb{O}$, the Cauchy--Kovalevskaya extension may be extended to octonion valued real analytic functions.  A necessary  condition for $\text{CK}[f]\in\mathcal{M}^{(\ell)}$ is that $f$ is an octonion valued real analytic function with $\m{f}\neq 0$. It is not a sufficient condition since, e.g.,
\[
\text{CK}[\m{x}](x)=7x_0+\m{x}
\]
belongs to $\mathcal{M}_{B}$.
\end{rem}


We may conclude that although Clifford and octonion analyses have formally very similar definitions, the corresponding function spaces are different.



\begin{prop}
$\mathcal{M}_R\subsetneq\mathcal{M}_B\subsetneq\mathcal{M}^{(\ell)}$
\end{prop}

\begin{proof}
The inclusions follow from Propositions \ref{prop:mtregular}--\ref{prop:kerdcl}. The examples showing that the inclusions are strict: if $f=x_2e_1-x_7e_4$, then $\cro f=f\cro=0$, but $\partial_{x_2}f_1-\partial_{x_1}f_2=1\ne0$, and if $f=x_1-x_2e_3$, then $\cro f=0$, but $f\cro=2e_1\ne0$.
\end{proof}

This result is crucial in understanding the fundamental character of octonion analysis and
the structural differences between octonion, quaternion, and Clifford analyses.

\begin{rem}[Quaternion analysis]
If we make the corresponding definitions for quaternion regular function classes by considering the Cauchy-Riemann operator $D_x=\partial_{x_0}+e_1\circ\partial_{x_1}+e_2\circ\partial_{x_2}+e_3\circ\partial_{x_3}$ acting on quaternion valued functions $f=f_1+f_1e_1+f_2e_2+f_3e_3$, then by comparing
\eqref{eq:rieszoctcomp} and \eqref{eq:rieszcomp} we observe immediately, that
\[
\mathcal{M}_R=\mathcal{M}_B\subsetneq\mathcal{M}^{(\ell)}.
\]
\end{rem}

\begin{rem}[Clifford analysis]
If we make the corresponding definitions for paravector valued monogenic functions, then
by Proposition \ref{prop:kerdcl},
\[
\mathcal{M}_R=\mathcal{M}_B=\mathcal{M}^{(\ell)}.
\]
\end{rem}


\section{Characterizing function classes in biaxial quaternion analysis}
In the preceding section we gave characterizations for left-, $B$-, and $R$-regular functions using componentwise and vector forms. In this section we write the three  systems of Section \ref{sec:MTBR} in quaternion forms. The use of the quaternion forms of the function and the Cauchy--Riemann operator is called the \emph{biaxial quaternion analysis}.

Write the Cauchy--Riemann operator $\cro$ and the function $f\colon\R^8\to\OO$ in the quaternion forms:
\begin{align*}
\cro&=\partial_{u_0}+\partial_{\m{u}}+(\partial_{v_0}+\partial_{\m{v}})\circ e_4,\\
f&=g_0+\m{g}+(h_0+\m{h})\circ e_4.
\end{align*}
According to Corollary \ref{cor:octprodqua}, we can write
\begin{align*}
\cro f=&\partial_{u_0}g_0-\partial_{v_0}h_0-\partial_{\m{u}}\cdot\m{g}-\partial_{\m{v}}\cdot\m{h}\\
&+\partial_{u_0}\m{g}+\partial_{\m{u}}g_0+\partial_{v_0}\m{h}-\partial_{\m{v}}h_0+\partial_{\m{u}}\times\m{g}-\partial_{\m{v}}\times\m{h}\\
&+(\partial_{u_0}h_0+\partial_{v_0}g_0+\partial_{\m{v}}\cdot\m{g}-\partial_{\m{u}}\cdot\m{h})e_4\\
&+(\partial_{u_0}\m{h}+\partial_{\m{v}}g_0+\partial_{\m{u}}h_0-\partial_{v_0}\m{g}-\partial_{\m{u}}\times\m{h}-\partial_{\m{v}}\times\m{g})\circ e_4.
\end{align*}

This implies the quaternion forms of the Moisil--Teodorescu type system \eqref{eq:mtsystem} and the system \eqref{eq:rieszoct}.

\begin{prop}\label{prop:lmquat}
$f\colon\R^8\to\OO$ is left regular if and only if it satisfies the system
\begin{equation}
\begin{aligned}
\partial_{u_0}h_0+\partial_{v_0}g_0+\partial_{\m{v}}\cdot\m{g}-\partial_{\m{u}}\cdot\m{h}&=0,\\
\partial_{u_0}g_0-\partial_{v_0}h_0-\partial_{\m{u}}\cdot\m{g}-\partial_{\m{v}}\cdot\m{h}&=0,\\
\partial_{u_0}\m{g}+\partial_{\m{u}}g_0+\partial_{v_0}\m{h}-\partial_{\m{v}}h_0+\partial_{\m{u}}\times\m{g}-\partial_{\m{v}}\times\m{h}&=0,\\
\partial_{u_0}\m{h}+\partial_{\m{v}}g_0+\partial_{\m{u}}h_0-\partial_{v_0}\m{g}-\partial_{\m{u}}\times\m{h}-\partial_{\m{v}}\times\m{g}&=0.
\end{aligned}
\end{equation}
\end{prop}

\begin{prop}\label{prop:lrmquat}
$f\colon\R^8\to\OO$ is left and right regular if and only if it satisfies the system
\begin{equation}
\begin{aligned}
\partial_{u_0}h_0+\partial_{v_0}g_0+\partial_{\m{v}}\cdot\m{g}-\partial_{\m{u}}\cdot\m{h}&=0,\\
\partial_{u_0}g_0-\partial_{v_0}h_0-\partial_{\m{u}}\cdot\m{g}-\partial_{\m{v}}\cdot\m{h}&=0,\\
\partial_{u_0}\m{g}+\partial_{\m{u}}g_0+\partial_{v_0}\m{h}-\partial_{\m{v}}h_0&=0,\\
\partial_{u_0}\m{h}+\partial_{\m{v}}g_0+\partial_{\m{u}}h_0-\partial_{v_0}\m{g}&=0,\\
\partial_{\m{u}}\times\m{g}-\partial_{\m{v}}\times\m{h}&=0,\\
\partial_{\m{u}}\times\m{h}+\partial_{\m{v}}\times\m{g}&=0.
\end{aligned}
\end{equation}
\end{prop}

One example of the use of the biaxial quaternion analysis is the proof of the following vector calculus identity in the octonionic case.

\begin{lem}
Let the coordinates of $f$ and $g:\mathbb{R}^8\to\OO$ have partial derivatives. Then
\[\crovec\cdot(\m{f}\times\m{g})=(\crovec\times\m{f})\cdot\m{g}-\m{f}\cdot(\crovec\times\m{g}).\]
\end{lem}

\begin{proof}
We use quaternion decompositions
\begin{align*}
\crovec&=\partial_{\m{u}}+\partial_{v_0}e_4+\partial_{\m{v}}\circ e_4,\\
\m{f}&=\m{f_1}+F_0e_4+\m{F_1}\circ e_4,\\
\m{g}&=\m{g_1}+G_0e_4+\m{G_1}\circ e_4.
\end{align*}
On the left-hand side we apply Lemma \ref{lem:crossquaform} to the cross product $\m{f}\times\m{g}$, and use the classical vector calculus identity
\[\nabla\cdot(\m{u}\times\m{v})=(\nabla\times\m{u})\cdot\m{v}-\m{u}\cdot(\nabla\times\m{v})\]
for $\m{u},\m{v}\colon\R^3\to\R^3$:
\begin{align*}
&\crovec\cdot(\m{f}\times\m{g})\\
&=\big(\partial_{\m{u}}+\partial_{v_0}e_4+\partial_{\m{v}}\circ e_4\big)\cdot\\
&\quad\big(F_0\m{G_1}-\m{F_1}G_0+\m{f_1}\times\m{g_1}-\m{F_1}\times\m{G_1}\\
&\quad+(\m{F_1}\cdot\m{g_1}-\m{f_1}\cdot\m{G_1})e_4\\
&\quad+(\m{f_1}G_0-F_0\m{g_1}-\m{f_1}\times\m{G_1}-\m{F_1}\times\m{g_1})\circ e_4\big)\\
&=
\partial_{\m{u}}\cdot(F_0\m{G_1})-\partial_{\m{u}}\cdot(\m{F_1}G_0)+\partial_{\m{u}}\cdot(\m{f_1}\times\m{g_1})-\partial_{\m{u}}\cdot(\m{F_1}\times\m{G_1})\\
&\quad+\partial_{v_0}(\m{F_1}\cdot\m{g_1})-\partial_{v_0}(\m{f_1}\cdot\m{G_1})\\
&\quad+\partial_{\m{v}}\cdot(\m{f_1}G_0)-\partial_{\m{v}}\cdot(F_0\m{g_1})-\partial_{\m{v}}\cdot(\m{f_1}\times\m{G_1})-\partial_{\m{v}}\cdot(\m{F_1}\times\m{g_1})\\
&=(\partial_{\m{u}} F_0)\cdot\m{G_1}+F_0(\partial_{\m{u}}\cdot\m{G_1})
-(\partial_{\m{u}}\cdot\m{F_1})G_0-\m{F_1}\cdot(\partial_{\m{u}}G_0)\\
&\quad+(\partial_{\m{u}}\times\m{f_1})\cdot\m{g_1}-\m{f_1}\cdot(\partial_{\m{u}}\times\m{g_1})
-(\partial_{\m{u}}\times\m{F_1})\cdot\m{G_1}+\m{F_1}\cdot(\partial_{\m{u}}\times\m{G_1})\\
&\quad+(\partial_{v_0}\m{F_1})\cdot\m{g_1}+\m{F_1}\cdot(\partial_{v_0}\m{g_1})
-(\partial_{v_0}\m{f_1})\cdot\m{G_1}-\m{f_1}\cdot(\partial_{v_0}\m{G_1})\\
&\quad+(\partial_{\m{v}}\cdot\m{f_1})G_0+\m{f_1}\cdot(\partial_{\m{v}}G_0)
-(\partial_{\m{v}}F_0)\cdot\m{g_1}-F_0(\partial_{\m{v}}\cdot\m{g_1})\\
&\quad-(\partial_{\m{v}}\times\m{f_1})\cdot\m{G_1}+\m{f_1}\cdot(\partial_{\m{v}}\times\m{G_1})
-(\partial_{\m{v}}\times\m{F_1})\cdot\m{g_1}+\m{F_1}\cdot(\partial_{\m{v}}\times\m{g_1}).
\end{align*}
On the right-hand side we apply Lemma \ref{lem:crossquaform} to the rotors $\crovec\times\m{f}$ and $\crovec\times\m{g}$:
\begin{align*}
&(\crovec\times\m{f})\cdot\m{g}\\
&=\big(\partial_{v_0}\m{F_1}-\partial_{\m{v}}F_0+\partial_{\m{u}}\times\m{f_1}-\partial_{\m{v}}\times\m{F_1}\\
&\quad+(\partial_{\m{v}}\cdot\m{f_1}-\partial_{\m{u}}\cdot\m{F_1})e_4\\
&\quad+(\partial_{\m{u}}F_0-\partial_{v_0}\m{f_1}-\partial_{\m{u}}\times\m{F_1}-\partial_{\m{v}}\times\m{f_1})\circ e_4\big)\cdot\\
&\quad\big(\m{g_1}+G_0e_4+\m{G_1}\circ e_4\big)\\
&=(\partial_{v_0}\m{F_1})\cdot\m{g_1}-(\partial_{\m{v}}F_0)\cdot\m{g_1}+(\partial_{\m{u}}\times\m{f_1})\cdot\m{g_1}-(\partial_{\m{v}}\times\m{F_1})\cdot\m{g_1}\\
&\quad+(\partial_{\m{v}}\cdot\m{f_1})G_0-(\partial_{\m{u}}\cdot\m{F_1})G_0\\
&\quad+(\partial_{\m{u}}F_0)\cdot\m{G_1}-(\partial_{v_0}\m{f_1})\cdot\m{G_1}-(\partial_{\m{u}}\times\m{F_1})\cdot\m{G_1}-(\partial_{\m{v}}\times\m{f_1})\cdot\m{G_1},
\end{align*}
and
\begin{align*}
&\m{f}\cdot(\crovec\times\m{g})\\
&=\big(\m{f_1}+F_0e_4+\m{F_1}\circ e_4\big)\cdot\\
&\quad\big(\partial_{v_0}\m{G_1}-\partial_{\m{v}}G_0+\partial_{\m{u}}\times\m{g_1}-\partial_{\m{v}}\times\m{G_1}\\
&\quad+(\partial_{\m{v}}\cdot\m{g_1}-\partial_{\m{u}}\cdot\m{G_1})e_4\\
&\quad+(\partial_{\m{u}}G_0-\partial_{v_0}\m{g_1}-\partial_{\m{u}}\times\m{G_1}-\partial_{\m{v}}\times\m{g_1})\circ e_4\big)\\
&=\m{f_1}\cdot(\partial_{v_0}\m{G_1})-\m{f_1}\cdot(\partial_{\m{v}}G_0)+\m{f_1}\cdot(\partial_{\m{u}}\times\m{g_1})-\m{f_1}\cdot(\partial_{\m{v}}\times\m{G_1})\\
&\quad+F_0(\partial_{\m{v}}\cdot\m{g_1})-F_0(\partial_{\m{u}}\cdot\m{G_1})\\
&\quad+\m{F_1}\cdot(\partial_{\m{u}}G_0)-\m{F_1}\cdot(\partial_{v_0}\m{g_1})-\m{F_1}\cdot(\partial_{\m{u}}\times\m{G_1})-\m{F_1}\cdot(\partial_{\m{v}}\times\m{g_1}).\qedhere
\end{align*}
\end{proof}

\begin{rem}[Regular functions is not a module]\label{rem:rfunnotmodule}
In quaternion analysis $\partial_ug=0$ implies $\partial_u(g\circ a)=0$ for all $a\in\HH$ (see Lemma \ref{lem:prodrulev}). The same does not hold in octonion analysis. For example, define $g\colon\HH\to\HH$, $g(x)=x_1-x_2e_3$. Then $\cro g=e_1-e_2e_3=0$, but $\cro(g\circ e_4)=\cro(x_1e_4-x_2e_7)=e_1e_4-e_2e_7=2e_5$. 
\end{rem}

For quaternion functions we have the product rules \eqref{eq:prodrulexv1} and \eqref{eq:prodrulexv2} for the Cauchy--Riemann operator.
Remark \ref{rem:rfunnotmodule} suggests that we do not have any kind of a non-trivial product rule for octonion valued
functions. To compute $D_x(fg)$ for octonion valued functions in practice, one way is to use biaxial quaternion analysis, and then to apply \eqref{eq:prodrulexv1}--\eqref{eq:prodrulexvvika}.

\begin{lem}{\cite[Thm 1.3.2]{GurlebeckSprozig}}
\label{lem:prodrulev}
Let the coordinates of $f$ and $g\colon\HH\to\HH$ have partial derivatives. 
Then
\begin{equation}\label{eq:prodrulexv1}
\partial_{\m{u}}(f\circ g)
=(\partial_{\m{u}}f)\circ g+\BAR{f}\circ(\partial_{\m{u}}g)-2(\m{f}\cdot\partial_{\m{u}})g
\end{equation}
and
\begin{equation}\label{eq:prodrulexv2}
(f\circ g)\partial_{\m{u}}
=(f\partial_{\m{u}})\circ\BAR{g}+f\circ(g\partial_{\m{u}})-2(\m{g}\cdot\partial_{\m{u}})f.
\end{equation}
Here, $(\m{f}\cdot\partial_{\m{u}})g=\sum\limits_{i=1}^3f_i\partial_{x_i}g$.
\end{lem}

\begin{cor}
Let the coordinates of $f$ and $g\colon\HH\to\HH$ have partial derivatives. Then
\begin{align}
\partial_{\m{u}}((f\circ e_4)\circ g)
&=[(f\partial_{\m{u}})\circ g+f\circ(\BAR{g}\partial_{\m{u}})+2(\m{g}\cdot\partial_{\m{u}})f]\circ e_4\\
\partial_{\m{u}}(f\circ(g\circ e_4))
&=[(g\partial_{\m{u}})\circ\BAR{f}+g\circ(f\partial_{\m{u}})-2(\m{f}\cdot\partial_{\m{u}})g]\circ e_4\\
\partial_{\m{u}}((f\circ e_4)\circ(g\circ e_4))
&=-(\partial_{\m{u}}\BAR{g})\circ f-g\circ(\partial_{\m{u}}f)-2(\m{g}\cdot\partial_{\m{u}})f\\
(\partial_{\m{v}}\circ e_4)(f\circ g)
&=[(\partial_{\m{v}}\BAR{g})\circ\BAR{f}+g\circ(\partial_{\m{v}}\BAR{f})+2(\m{g}\cdot\partial_{\m{v}})\BAR{f}]\circ e_4\\
(\partial_{\m{v}}\circ e_4)((f\circ e_4)\circ g)
&=-(g\partial_{\m{v}})\circ f-g(\BAR{f}\partial_{\m{v}})-2(\m{f}\cdot\partial_{\m{v}})g\\
(\partial_{\m{v}}\circ e_4)(f\circ(g\circ e_4))
&=-(\BAR{f}\partial_{\m{v}})\circ g-\BAR{f}\circ(\BAR{g}\partial_{\m{v}})-2(\m{g}\cdot\partial_{\m{v}})\BAR{f}\\
(\partial_{\m{v}}\circ e_4)((f\circ e_4)\circ(g\circ e_4))
&=[-(\partial_{\m{v}}\BAR{f})g-f(\partial_{\m{v}}g)-2(\m{f}\cdot\partial_{\m{v}})g]\circ e_4\label{eq:prodrulexvvika}
\end{align}
\end{cor}

\begin{proof}
Apply Lemmas \ref{lem:quaoctprod} and \ref{lem:prodrulev}, and use the fact $\BAR{fg}=\BAR{g}\,\BAR{f}$.
\end{proof}

\section{Characterization of the function classes in Clifford analysis}
In this last section, we study
the classes of  left-, $B$-, and $R$-regular functions using Clifford analysis.
We begin with the following algebraic lemma.

\begin{lem}\label{lem:abIcl}
Let $I=I^-$ be the primitive idempotent \eqref{eq:primidef}, and let $a=a_0+\m{a}$ and $b=b_0+\m{b}\in\Cl_{0,7}$ be paravectors. Then
\begin{align}
16[abI]_0&=a_0b_0-\m{a}\cdot\m{b},\label{eq:16ideeka}\\
16[abI]_1&=a_0\m{b}+\m{a}b_0-[(\m{a}\wedge\m{b})W]_1,\label{eq:16idetoka}\\
16[abI]_2&=\m{a}\wedge\m{b}-[(a_0\m{b}+\m{a}b_0)W]_2+[(\m{a}\wedge\m{b})We_{1\cdots7}]_2,\\
16[abI]_3&=-(a_0b_0-\m{a}\cdot\m{b})W+[(a_0\m{b}+\m{a}b_0)We_{1\cdots7}]_3-[(\m{a}\wedge\m{b})W]_3,\\
16[abI]_4&=(a_0b_0-\m{a}\cdot\m{b})We_{1\cdots7}-[(a_0\m{b}+\m{a}b_0)W]_4+[(\m{a}\wedge\m{b})We_{1\cdots7}]_4,\\
16[abI]_5&=[(a_0\m{b}+\m{a}b_0)We_{1\cdots7}]_5-
[(\m{a}\wedge\m{b})W]_5-(\m{a}\wedge\m{b})e_{1\cdots7},\\
16[abI]_6&=-(a_0\m{b}+\m{a}b_0)e_{1\cdots7}+[(\m{a}\wedge\m{b})We_{1\cdots7}]_6,\\
16[abI]_7&=-(a_0b_0-\m{a}\cdot\m{b})e_{1\cdots7},\label{eq:16idevika}
\end{align}
and
\begin{equation}\label{eq:idesym}
[abI]_k=0\ \Leftrightarrow\ [abI]_{7-k}=0,\quad k=0,1,\ldots,7.
\end{equation}
If $[abI]_0=0$, then the conditions $[abI]_j=0$, $j=2,3,4,5$, are pairwise equivalent. In particular, if $[abI]_{0,1,2}=0$, then $abI=0$.
\end{lem}

\begin{proof}
Write the real part and $1$- and $2$-vector parts of $ab$ using \eqref{eq:clparavect}, and expand the definition \eqref{eq:primidef} of $I$ using the fact $e_{12\cdots7}^2=1$:
\begin{align*}
ab&=(a_0b_0-\m{a}\cdot\m{b})+(a_0\m{b}+\m{a}b_0)+\m{a}\wedge\m{b},\\
16I&=1-W+We_{12\cdots7}-e_{12\cdots7}.
\end{align*}
Here, $W$ is a $3$-vector and $We_{12\cdots7}$ is a $4$-vector. Then, for example, $\m{a}W$ only contains $2$- and $4$-vector parts, and therefore $[\m{a}W]_3=0$. This kind of reasoning implies \eqref{eq:16ideeka}--\eqref{eq:16idevika}.

Now, \eqref{eq:idesym} follows from the facts that for any $c\in\Cl_{0,7}$,
\begin{align*}
c=0\ &\Leftrightarrow\ ce_{12\cdots7}=0,\quad\text{and}\\
[c]_ke_{12\cdots7}&=[ce_{12\cdots7}]_{7-k},\quad k=0,1,\ldots,7.
\end{align*}
To prove the last claim, it is now enough to show that in the case $[abI]_0=0$, $[abI]_2=0$ if and only if $[abI]_3=0$. This can be seen by computing
\begin{align*}
16[abI]_2
=&(a_0b_1+a_1b_0+a_2b_3-a_3b_2+a_4b_5-a_5b_4-a_6b_7+a_7b_6)(e_{23}+e_{45}-e_{67})\\
+&(a_0b_2-a_1b_3+a_2b_0+a_3b_1+a_4b_6+a_5b_7-a_6b_4-a_7b_5)(-e_{13}+e_{46}+e_{57})\\
+&(a_0b_3+a_1b_2-a_2b_1+a_3b_0+a_4b_7-a_5b_6+a_6b_5-a_7b_4)(e_{12}+e_{47}-e_{56})\\
+&(a_0b_4-a_1b_5-a_2b_6-a_3b_7+a_4b_0+a_5b_1+a_6b_2+a_7b_3)(-e_{15}-e_{26}-e_{37})\\
+&(a_0b_5+a_1b_4-a_2b_7+a_3b_6-a_4b_1+a_5b_0-a_6b_3+a_7b_2)(e_{14}-e_{27}e_{36})\\
+&(a_0b_6+a_1b_7+a_2b_4-a_3b_5-a_4b_2+a_5b_3+a_6b_0-a_7b_1)(e_{17}+e_{24}-e_{35})\\
+&(a_0b_7-a_1b_6+a_2b_5+a_3b_4-a_4b_3-a_5b_2+a_6b_1+a_7b_0)(-e_{16}+e_{25}+e_{34}),
\end{align*}
and in the case $[abI]_0=0$,
\begin{align*}
16[abI]_3
=&(a_0b_1+a_1b_0+a_2b_3-a_3b_2+a_4b_5-a_5b_4-a_6b_7+a_7b_6)(e_{247}-e_{256}-e_{346}-e_{357})\\
+&(a_0b_2-a_1b_3+a_2b_0+a_3b_1+a_4b_6+a_5b_7-a_6b_4-a_7b_5)(-e_{147}+e_{156}+e_{345}-e_{367})\\
+&(a_0b_3+a_1b_2-a_2b_1+a_3b_0+a_4b_7-a_5b_6+a_6b_5-a_7b_4)(e_{146}+e_{157}-e_{245}+e_{267})\\
+&(a_0b_4-a_1b_5-a_2b_6-a_3b_7+a_4b_0+a_5b_1+a_6b_2+a_7b_3)(e_{127}-e_{136}+e_{235}-e_{567})\\
+&(a_0b_5+a_1b_4-a_2b_7+a_3b_6-a_4b_1+a_5b_0-a_6b_3+a_7b_2)(-e_{126}-e_{137}-e_{234}+e_{467})\\
+&(a_0b_6+a_1b_7+a_2b_4-a_3b_5-a_4b_2+a_5b_3+a_6b_0-a_7b_1)(e_{125}+e_{134}-e_{237}-e_{457})\\
+&(a_0b_7-a_1b_6+a_2b_5+a_3b_4-a_4b_3-a_5b_2+a_6b_1+a_7b_0)(-e_{124}+e_{135}+e_{236}+e_{456}).\qedhere
\end{align*}
\end{proof}


We infer that left-, B-, and R-regularity can be studied by considering paravector-spinor valued functions $fI$.

\begin{thm}
Suppose $f\colon\mathbb{R}^8\to\mathbb{R}^8$ is a paravector valued function such that the coordinate functions have partial derivatives.
\begin{enumerate}[(a)]
\item
$f$ is left-regular if and only if
\begin{equation}
[\partial_xfI]_j=0\ \text{for } j=0,1.
\end{equation}
\item
$f$ is B-regular if and only if
\begin{equation}
[\partial_xfI]_j=0\ \text{for } j=0,1,
\quad\text{and}\quad
[\partial_xfW]_1=0.
\end{equation}
\end{enumerate}
\end{thm}

\begin{proof}
(a) follows using Lemma \ref{lem:16abi}:
\begin{align*}
\cro f=16[\partial_x fI]_{0}+16[\partial_x fI]_{1}.
\end{align*}
(b) We compute, using \eqref{eq:clparavect} and \eqref{eq:16idetoka},
\begin{align*}
[\partial_xfW]_1&=[(\partial_{x_0}f_0-\partial_{\m{x}}\cdot\m{f})W]_1+[(\partial_{x_0}\m{f}+\partial_{\m{x}}f_0)W]_1+[(\partial_{\m{x}}\wedge\m{f})W]_1\\
&=[(\partial_{\m{x}}\wedge\m{f})W]_1\\
&=-16[\partial_xfI]_1+\partial_{x_0}\m{f}+\crovec f_0.
\end{align*}
Since $\crovec\times\m{f}=-[(\partial_{\m{x}}\wedge\m{f})W]_1$ (Lemma \ref{lem:abW}),
the claim now follows from (a) and Propositions \ref{prop:mtregular}--\ref{prop:rregular}.
\end{proof}

\begin{rem}
If $\partial_xf=0$, then (trivially) $[\partial_xfI]_j=0$ for all $j=0,1,\ldots,7$. The converse does not hold. This follows from the fact that the equation $aI=0$ does not have a unique solution $a=0$ in the Clifford algebra. Hence, paravector spinor valued solutions to the Cauchy-Riemann equations forms a bigger function class, and the class of R-regular solutions is 
\[
\mathcal{M}_R
\subsetneq
\{f : \partial_xfI=0\}
=\{f : [\partial_xfI]_j=0,\ j=0,1,\ldots,7\}
=\{f : [\partial_xfI]_j=0,\ j=0,1,2\}.
\]
Equality of the latter two function classes follows from 
Lemma \ref{lem:abIcl}.
An example showing that the inclusion is strict: if $f=x_2e_1-x_7e_4$, then $\partial_xf=e_4e_7-e_1e_2$, but $[\partial_xfI]_j=0$ for $j=0,1,2$. 
\end{rem}

\section*{Conclusion}
The idea of this paper is to study differences between octonion and Clifford analyses. This leads us to observe the fundamental difference between octonion regular and Clifford monogenic functions. The structure of octonion regular functions is studied by comparing left-, right-, $B$-, and $R$-regular functions. The existence of these classes is a consequence of different algebraic properties of the algebras. In the heart of octonion analysis is the study of the properties of these function classes and their relations, which distinguishes it essentially from Clifford analysis.

\end{document}